\documentclass{article}
 
\usepackage{amssymb} 
\usepackage{latexsym}
\usepackage{amsmath} 
\usepackage{amsthm}
\usepackage{enumitem}
\usepackage{geometry}

\title{On the linearizability of 3-webs: end of controversy}

\author{Zolt\'an Muzsnay}

\date{December 10, 2017}

\newcommand{\W}{{\mathcal W}_{\scriptscriptstyle 0}}

\newtheorem*{theorem*}{Theorem}
\newtheorem*{remark*}{Remark}

\begin{document}

\maketitle

\begin{abstract}  
  There are two theories describing the linearizability of 3-webs: one is
  developed in \cite{GMS} and another in \cite{GL}.  Unfortunately they cannot
  be both correct because on an explicit 3-web $\W$ they \emph{contradict}:
  the first predicts that $\W$ is linearizable while the second states that
  $\W$ is not linearizable.  The essential question beyond this particular
  3-web is: which theory describes correctly the linearizability condition?
  In this paper we present a very short proof, due to J.-P.~Dufour, that $\W$
  is linearizable, confirming the result of \cite{GMS}.  \medskip
\end{abstract}

\medskip

\begin{description}
\item[AMS Classification:] 53A60,
\item[Keywords:] Webs, linearization.
\end{description}

\footnotetext{This work is partially supported by the EFOP-3.6.2-16-2017-00015
  project and by the EFOP-3.6.1-16-2016-00022 project.  The projects have been
  supported by the European Union, co-financed by the European Social Fund.}

\section{The linearizability problem for planar 3-webs}

On a two-dimensional real or complex differentiable manifold $M$ a $3$-web is
given by $3$ foliations of smooth curves in general position.  Two webs
${\mathcal W}$ and $\widetilde{{\mathcal W}}$ are locally equivalent at
$p\in M$, if there exists a local diffeomorphism on a neighborhood of $p$
which exchanges them.  A $3$-web is called \textit{linear} if it is given by
$3$ foliations of straight lines. A web which is equivalent to a linear web is
called linearizable.

\smallskip

\textbf{The linearizability problem}: \emph{Characterize the $3$-webs on real
  or complex $2$-dimen\-si\-onal manifolds which are equivalent, up to a local
  diffeomorphism, to {linear} webs, that is webs such that the corresponding
  foliations are straight lines in a convenient coordinate system.}

\smallskip

Similar to the linearizability is the notion of parallelizability.  A $3$-web
is called \textit{parallelizable} if it is equivalent to $3$ families of
\textit{parallel} lines.  One can remark that for 1- and 2-webs the notion of
linearizability and parallelizability coincide: Because of the inverse function
theorem, any $1$- and $2$-webs are linearizable and also parallelizable.  This
is not true in general: the notion of parallelizability is much stronger than
the linearizability.  A generic $3$-web is non-linearizable, and even if a web
is linearizable, it is in general non-parallelizable.

Basic examples of planar 3-webs comes from complex projective algebraic
geometry. If $\mathcal C \subset \mathbb P^2$ is a not necessarily irreducible
and possibly singular algebraic curve of degree 3 on the projective plane
$\mathbb P^2$, then by duality in the Grassmannian manifold
$\mathbb Gr(1, \mathbb P^2)$ one can obtain a 3-web called the algebraic web
associated with $\mathcal C \subset \mathbb P^2$
(cf.~\cite{Pereira_Pirio_2015}).  Graf and Sauer proved a theorem, which in
web geometry language can be stated as follows: a linear web is parallelizable
if and only if it is associated to an algebraic curve of degree 3, i.e.~its
leaves are tangent lines to an algebraic curve of degree 3 \cite[page 24]{BB}.
This theorem is a special case of N.H.~Abel's classical theorem and its
converse: the general Lie-Darboux-Griffiths theorem \cite{Griffiths}.

Concerning the parallelizability of 3-webs, an elegant coordinate free
characterization can be given in terms of the Chern connection associated: a
$3$-web is parallelizable if and only if the curvature of the Chern
connection, called also Blaschke curvature, vanishes \cite{Chern_1935}.  A new
theoretical set-up of the problem can be found in \cite{Nagy_1988}.

Although the problem of finding a linearizability criterion is a very natural
one, it is far from being trivial.  T.H.~Gronwall conjectured that if a
non-parallelizable $3$-web ${\mathcal W}$ is linearizable, then up to a
projective transformation there is a unique diffeomorphism which maps
${\mathcal W}$ into a linear $3$-web.  G.~Bol suggested a method in
\cite{Bol1} how to find a criterion of linearizability, but he was unable to
carry out the computation.  He showed that the number of projectively
different linear $3$-webs in the plane which are equivalent to a
non-parallelizable $3$-web is finite and less that 17.  The formulation of the
linearizability problem in terms of the Chern connection was suggested by
M.A.~Akivis in a lecture given in Moscow in 1973. In his approach the
linearizability problem is reduced to the solvability of a system of nonlinear
partial differential equations on the components of the affine deformation
tensor. Using Akivis' idea V.V.~Goldberg determined in \cite{Gol2} the first
integrability conditions of the partial differential system.

\section{The controversy}

In 2001 J.~Grifone, Z.~Muzsnay and J.~Saab solved the linearizability problem
by carrying out the computation \cite{GMS}. They showed that, in the
non-parallelizable case, there exists an algebraic submanifold ${\mathcal A}$
of the space of vector valued symmetric tensors ($S^2T^*\otimes T$) on a
neighborhood of any point $p\in M$, expressed in terms of the curvature of the
Chern connection and its covariant derivatives up to order $6$, so that the
affine deformation tensor is a section of $S^2T^*\otimes T$ with values in
${\mathcal A}$.  In particular: the web is linearizable if and only if
${\mathcal A}\neq \emptyset$ and there exists at most $15$ projectively
nonequivalent linearizations of a nonparallelizable 3-web.  The expressions of
the polynomials and their coefficients which define $\mathcal{A}$ can be found
in \cite{GMS_publi}. The criteria of linearizability provides the possibility to
make explicit computation on concrete examples to decide whether or not they
are linearizable.

In 2006 V.V.~Goldberg and V.V.~Lychagin found results on the linearizability
in \cite{GL}. Their results were different from that of \cite{GMS} and they
qualified \cite{GMS} ``incomplete because they do not contain all conditions''
(see \cite[page 171]{GL_cras} and \cite[page 70]{GL}) without pointing out any
missing integrability condition or developing any further justification.

The GMS-approach developed in \cite{GMS} and the GL-approach described in
\cite{GL} \emph{cannot be both correct} because there are cases where the two
theories \emph{contradict}.  

Hence the small but dedicated scientific community working on the problems
related to web geometry is in suspense (see for example \cite[page
2]{Agafonov2015}, \cite[page 2]{Agafonov2017}, or \cite[page 40]{Wang_2012}).
Therefore the focus of this paper is to conclude which theory is describing
correctly the linearizability condition.

\section{Decisive example}

The direct comparison of the two theories is not straightforward, since the
formulas in both cases are long and complex containing the curvature tensor
and its different derivatives.  There is, however, a very specific case, where
the two theories show clearly opposite results. This explicit example of 3-web
was described in \cite{GMS}.  The particular 3-web $\W$ is determined by the web
function $f(x,y):=(x+y)e^{-x}$, i.e.~it is the 3-web given by the foliations
\begin{equation}
  \label{eq:example}
  x=const, \qquad y = const, \qquad (x+y) \,e^{-x}=const,
\end{equation}
on the domain $D:=\{(x,y) \, | \, x+y \neq 1\}\subset \mathbb R^2$.  Using the
GMS--theory one gets that ${\mathcal W}$ is linearizable (page 2653,
\cite{GMS}) while GL--theory states the opposite (page 171, line 7--10,
\cite{GL}).  Evidently, \emph{the correct theory should give a correct answer
  in that specific situation}. In the theorem below we show that the web $\W$
is linearizable, therefore the prediction of GMS--theory is correct.  This
result was obtained in \cite{GMS} but the very short proof is due to
J.-P. Dufour. 
\begin{theorem*}
  \label{thm}
  The 3-web $\W$ defined by the foliations $x=const$, $y = const$ and
  $f(x,y):=(x+y) \,e^{-x}=const$, is linearizable.
\end{theorem*}
\begin{proof}
  The change of variable $\bar{x}= f(x, y)$, $\bar{y}= y$ clearly transforms
  the foliations $y = const$ and $f (x, y) = const$ into linear
  foliations. The line $x = c$ of the first foliation becomes the line
  $\bar{x} = (c + \bar{y})e^{-c}$.
\end{proof}
\begin{remark*}
  The statement of the theorem remains true if the function $f(x,y)$ has the
  form
  \begin{math}
    f(x,y)=a(x)x+b(x)y.
  \end{math}
\end{remark*}

\medskip

We note that the linearizability of $\W$ has already been investigated in
\cite{MZ08} from a different point of view: it was showed, using the GMS
approach, that $\W$ is linearizable by proving the existence of the affine
deformation tensor. The lack of presenting the explicit linearization map,
however, could maintain in some way the suspense. Now the suspense is over:
using the Theorem  we can conclude that the prediction of GMS-theory is
correct and the statement of the GL-theory is wrong. One can also conclude
that the criterion of linearizability of \cite{GMS, GMS_publi} provides
effective tools to decide whether or not a 3-web is linearizable.

\bigskip

\textbf{Acknowledgement.} The author is grateful to Professor J.P.~Dufour, who
indicated the linearizing diffeomorphism in a personal communication.  The
author would like to thank the referee for the constructive comments and
recommendations which contributed to improving the paper.

\bigskip

\date{\
  \\
  \textsc{Zolt\'an Muzsnay}, \vspace{1pt}
  \\
  Institute of Mathematics, University of Debrecen,
  \\
  H-4002 Debrecen, Pf. 400, Hungary,
  \\
  \emph{E-mail:} {\tt {}muzsnay@science.unideb.hu}
  \\
  \emph{url:} {\tt http://math.unideb.hu/muzsnay-zoltan/en}}
\end{document}